\documentclass{article}
\usepackage[usenames,dvipsnames]{xcolor}
\usepackage{amsmath, amssymb, amsfonts, color, stmaryrd}

\newcommand{\be}{\begin{equation}}
\newcommand{\ee}{\end{equation}}

\newcommand{\ba}{\begin{align*}}
\newcommand{\eal}{\end{align*}}

\newcommand{\lb}{\llbracket}
\newcommand{\rb}{\rrbracket}
\newcommand{\NNM}{\mathbb{N}}
\newcommand{\RRM}{\mathbb{R}}

\newcommand{\cle}{{\mathcal E}}

\newcommand{\clz}{{\cal Z}}

\newcommand{\clr}{{\cal R}}

\newcommand{\vdiv}{{\mathop{\mbox{\rm div\,}}}}

\newcommand{\clp}{{\cal P}}

\newcommand{\cln}{{\cal N}}

\newcommand{\bea}{\begin{eqnarray*}}
\newcommand{\eea}{\end{eqnarray*}}
\newcommand{\eps}{\varepsilon}

\renewcommand{\sup}{\mathop{\mbox{\rm sup}}}

\renewcommand{\dim}{\mathop{\mbox{\rm dim}}}

\newtheorem{theorem}{Theorem}[section]
\newtheorem{lemma}[theorem]{Lemma}
\newtheorem{corollary}[theorem]{Corollary}
\newtheorem{definition}[theorem]{Definition}

\newenvironment{proof}
 {\begin{trivlist}
    \item[\hskip\labelsep{\sc Proof:}\quad]}
    {\hspace*{\fill}\rule{2mm}{2mm}
  \end{trivlist}}

\newenvironment{keywords}%
   {\begin{trivlist}\item[]{\bfseries\sffamily Keywords:}\ }% oder "Keywords:"
   {\end{trivlist}}

\newenvironment{MSC}%
   {\begin{trivlist}\item[]{\bfseries\sffamily MSC:}\ }% oder  
   {\end{trivlist}}

\numberwithin{equation}{section}

\begin{document}

\title{Stability of Equilibria of a two-phase Stokes-Osmosis problem}
\author{Friedrich Lippoth \\ \small{Institute of Applied Mathematics, Leibniz 
University Hannover, Welfengarten 1,} \\ \small{ D-30167 Hannover, Germany }
\\ \small{e-mail: \texttt{lippoth@ifam.uni-hannover.de}} \\
%\\Mark A. Peletier\\ 
\\Georg Prokert\\
\small{Faculty of Mathematics and Computer
Science, TU Eindhoven} \\ \small{P.O. Box 513 5600 MB Eindhoven, the
Netherlands} \\
\small{e-mail: \texttt{g.prokert@tue.nl}}
}

\date{}
\maketitle

\begin{center}
\end{center}

\begin{abstract}
Within the framework of variational modelling we derive a two-phase moving boundary
problem that describes the motion of a semipermeable membrane separating two viscous 
liquids in a fixed container. The model includes the effects of osmotic pressure and 
surface tension of the membrane. For this problem we prove that the manifold of steady 
states is locally exponentially attractive.

\end{abstract}


\begin{MSC}
35R37, 35B35, 35K55, 76D07, 76M30 
\end{MSC}

\section{Introduction} \label{int}
This paper is devoted to a two-phase moving boundary problem describing  osmotic 
swelling of a closed membrane in a viscous liquid. 

Let $C \subset \mathbb{R}^N$ ($N \geq 2$) be a bounded connected open set  with 
smooth boundary representing a fixed region filled with an incompressible 
viscous liquid that moves according to the  
velocity field $u = 
u(t,x)$. Inside the liquid there is a closed connected semipermeable 
membrane $\Gamma(t) \subset C$ enclosing an open set $\Omega_+(t)$ and 
separating it from the outer phase $\Omega_-(t) := C \setminus 
\bar{\Omega}_+(t)$. In both phases a certain amount of a solute is dissolved. 
Its scalar concentration $c = c(t,x)$ evolves by convection along $u$ and 
diffusion through the liquid. It may be discontinuous across $\Gamma(t)$. Both 
the diffusivities and the viscosities are assumed to be constant and positive 
but possibly different in $\Omega_+$ and $\Omega_-$, respectively. 

The membrane is permeable for the liquid but  impermeable for the solute. Its 
deformation and movement are governed by surface tension forces, osmotic 
pressure, and the fluid motion. Based on these assumptions the 
following 
moving boundary problem can be derived using the approach of variational 
modelling, see Section 2:
\be \label{sos}
\left.\begin{array}{rclll}
-\nu_\pm \Delta u_\pm + \nabla(q_\pm + c_\pm) & = & 0 &  \mbox{ in 
$\Omega_\pm(t)$,} & t > 0,\\
\vdiv u_\pm & = & 0 & \mbox{ in $\Omega_\pm(t)$,} & t > 0,\\
\lb \tau(u,q + c) \rb n & = & H n & \mbox{ on $\Gamma(t)$,} & t > 0,\\
\lb u \rb & = & 0 & \mbox{ on $\Gamma(t)$,} & t > 0,\\
u_- & = & 0 & \mbox{ on $\partial C$,} & t > 0, \\
\\
\partial_t c_\pm - \kappa_\pm  \Delta c_\pm + \nabla c_\pm \cdot u_\pm & = & 0 & 
\mbox{ in $\Omega_\pm(t)$,} & t > 0,\\
\kappa_\pm \partial_n c_\pm + c_\pm (\lb c \rb + H)  & = & 0 & \mbox{ on 
$\Gamma(t)$,} & t > 0, \\
\partial_n c_- & = & 0 & \mbox{ on $\partial C$,} & t > 0, \\
\\
V_n & = & H + \lb c \rb + u \cdot n & \mbox{ on $\Gamma(t)$,} & t > 0, \\
\\
\Gamma(0) & = & \Gamma^0, & & \\
c_\pm(0) & = & c_\pm^0 & \mbox{ in $\Omega_\pm(0)$,} & 
\end{array}\right\}
\ee
where we used the  notation $u_\pm := u|_{\Omega_\pm}$, $c_\pm := 
c|_{\Omega_\pm}$. The brackets $\lb \cdot \rb$  indicate the jump of a quantity across 
$\Gamma(t)$, i.e. 
\[
\lb w(t,\cdot) \rb(x) :=  \lim_{y \in \Omega_+(t), y \rightarrow x} w(t,y) 
- \lim_{y \in \Omega_-(t), y \rightarrow x} w(t,y)
\]
for $x \in \Gamma(t)$ and $w: \Omega_+(t) \cup \Omega_-(t) \rightarrow 
\mathbb{R}$. Further, $H = H(t,x)$ denotes the $(N-1)$ - fold mean curvature of 
the closed compact hypersurface $\Gamma(t) = \partial \Omega_+(t)$ at the point 
$x \in \Gamma(t)$, oriented in the way that spheres have negative curvature, 
while $V_n$ is the normal velocity of the family $\{\Gamma(t)\}$ w.r.t the unit 
normal field $n = n(t)$ of $\Gamma(t)$ pointing outward $\Omega_+(t)$. The 
operator $\partial_n$ takes the directional derivative of a sufficiently regular 
function w.r.t the normal field $n(t)$. If no confusion seems likely, we use the 
same symbol $\partial_n$ to denote the derivative in the direction normal to 
$\partial C$ and exterior to $C$ as well. The symbol $q = q(t,x)$ stands for the 
hydrodynamic pressure and 
\[
\tau_\pm(u_\pm,q) := \nu^\pm \eps(u_\pm) - q_\pm \, \mbox{Id}  := 
\nu^\pm (\nabla u_\pm + (\nabla u_\pm)^T) - q_\pm \, \mbox{Id}, \quad 
q_\pm := q|_{\Omega_\pm}, 
\]
is the hydrodynamic stress tensor.  Note that the initial velocity $u(0)$ is 
uniquely determined by $c(0)$ and $\Gamma(0)$ as we shall discuss later in some 
detail, cf. Section \ref{stokes}.

System (\ref{sos}) is written in dimensionless form. The given positive  
constants $\kappa_\pm$ and $\nu_\pm$ carry information about physical 
parameters 
such as diffusivity of the solute, viscosity of the liquid in both phases and 
permeability of the membrane to solvent. For later use we introduce the 
piecewise constant functions 
\[
\nu(x,t) := \left\{ \begin{array}{ll} \nu_+ & x \in \Omega_+(t),  \\ \nu_- & x 
\in \Omega_-(t), \end{array} \right. \qquad \kappa(x,t) := \left\{ 
\begin{array}{ll} \kappa_+ & x \in \Omega_+(t), \\ \kappa_- & x \in 
\Omega_-(t). 
\end{array} \right.
\]
In the corresponding one-phase situation, a detailed derivation  of 
the model within the framework of variational modelling has been given in
\cite{lpp}. The two-phase problem is obtained in a parallel fashion. 
Therefore we restrict ourselves here to a brief recapitulation of the 
chosen setup which will be given in Section \ref{modsec}.

The paper \cite{lpp} also contains a short-time existence result for 
classical solutions for the one-phase problem. For the simpler limit problem in 
which the membrane moves through an immobile liquid the 
existence of classical solutions for a short time has been established in 
\cite{lipr1}, and the paper \cite{lipr2} deals with a stability analysis of its 
equilibria. For different modelling approaches (excluding fluid motion) as 
well as analytic results in even more special situations such as radial symmetry 
we refer to \cite{Pi, rub, Ve92, Ve00, zaal1, zaal2} and the references given in 
\cite{lpp}.

In this paper we focus on the equilibria of  system (\ref{sos}) and their 
stability properties. These equilibria form a finite dimensional 
submanifold of the phase space. Our main result states that this manifold is 
locally exponentially attractive, i.e the system is normally stable in the 
sense of  \cite{PSZa}.

While the main line of the proof is parallel to the one in \cite{lipr2}, 
we have to deal with the additional difficulty of handling the nonlocal 
solution operator of the two-phase Stokes system (\ref{sos})$_1$ - 
(\ref{sos})$_5$. In particular, the results of \cite{dpz} that are a crucial 
ingredient of the stability analysis in \cite{lipr2, PSZ} 
are no longer directly applicable. Additionally, one has to discuss the 
full two-phase Stokes system with respect to well-posedness and regularity. As 
these results do not seem to be readily and explicitly available in the 
literature, we include a proof of them in an appendix.

The present paper is organized as follows. In Section \ref{modsec} we 
explain briefly how our model can be derived within the framework of 
variational modelling. Section \ref{eq} identifies the equilibria 
of system (\ref{sos}). In Section \ref{lineq} we transform the problem to a 
fixed reference domain, determine the linearization of the transformed 
problem around an arbitrary fixed equilibrium and analyse some spectral  
properties of the corresponding linear operator. In this section we also give a 
precise formulation of our main result (Theorem \ref{MTh}), 
which is proved in Section \ref{sol}. The appendix (Section \ref{app}) contains 
a detailed discussion of the full two-phase Stokes system 
(\ref{2psC1}) and some abstract facts that are helpful for the 
spectral analysis (Lemma \ref{spectrum}, Corollary \ref{gap}, Lemma 
\ref{decomp}). 
\section{Modelling} \label{modsec}
We use the same modelling approach as in \cite{lpp} and derive  our model 
from the following building blocks:
\begin{itemize} 
\item[i)] We consider paths in a {\em state manifold} $\clz$  consisting of 
pairs $(\Omega_+,c)$ of a simply connected domain $\Omega_+$ satisfying 
$\bar{\Omega}_+ \subset C$ and a nonnegative solute concentration $c:\bar C 
\rightarrow \mathbb{R}$ that may be discontinuous across $\partial \Omega_+$. 
The domain $\Omega_+$ and the container $C$ uniquely determine $\Omega_-$. 
\item[ii)] On $\clz$ we define the {\em energy functional} 
\begin{equation} \label{ef}
\cle(\Omega_+,c) := \gamma \int_{C} c\ln c\,dx + \alpha |\partial\Omega_+| 
\end{equation}
with positive constants $\alpha$ and $\gamma$, cf. \cite{lpp} for  a more 
detailed discussion of their physical meaning. This choice includes diffusion of 
the solute and surface tension of the membrane as driving mechanisms of the 
evolution.
\item[iii)] The processes that dissipate energy are solvent motion,  solute 
flux, and passage of solvent through the membrane. Taking into account 
incompressibility of the solvent and mass conservation of the solute, these 
processes can be represented by triples 
\[
\{(u,f,V_n)\,|\,\vdiv u=0 \mbox{ in $C$, } \lb u \rb = 0  \mbox{ on $\partial 
\Omega_+$, } f_\pm \cdot n = c_\pm V_n\mbox{ on $\partial \Omega_+$}\}           
\]
which we collect in the {\em process space} $\clp_{(\Omega_+,c)}$. 
\item[iv)] The {\em dissipation functional} is defined on $\clp_{(\Omega_+,c)}$  
and given by   
\[
\Psi_{(\Omega_+,c)}(u,f,V_n)  :=  \frac{1}{2}\int_{C} 
\frac{\nu_1 |f-cu|^2}{c}\,dx +  \frac{1} {2}\int_{C} \nu_2 |\eps(u)|^2\,dx  
 + \; \frac{\nu_3}{2}\int_{\partial\Omega_+}(u\cdot n-V_n)^2\,d\sigma, 
\]
where $\nu_j = \nu_j^\pm$ in $\Omega_\pm$ ($j = 1,2$) and $\nu_3$ are  positive 
constants related to mobility of the solute, viscosities of the solvent in both 
phases and to the membrane's permeability to the solvent, cf. again \cite{lpp}. 
\item[v)] Observe that the elements of the tangent spaces $T_{(\Omega_+,c)}  
\clz$ of the state manifold $\clz$ can be represented by pairs $(V_n,\dot{c})$, 
where $V_n:\;\partial\Omega_+\longrightarrow\RRM$ is a normal velocity and 
$\dot{c}$ is a concentration change. Since mass conservation of the solute is 
expressed by the relation $\dot{c} + \mbox{div}f = 0$, it seems natural to 
define the {\em process map} $\Pi_{(\Omega_+,c)}: \clp_{(\Omega_+,c)} 
\rightarrow T_{(\Omega_+,c)} \clz$ defined by 
\[
\Pi_{(\Omega_+,c)}(u,f,V_n)=(V_n,-\vdiv f). 
\]
\end{itemize}
The model (\ref{sos}) is now determined by the dynamical system on $\clz$  
\begin{equation}\label{varmin1}
\dot z=\Pi_z w^\ast, \qquad z = (\Omega_+,c),
\end{equation}
where $w^\ast$ is the solution to the minimization problem
\begin{equation}\label{varmin2}
\Psi_z(w)+\cle'(z)[\Pi_z w]\;\longrightarrow\;\min,\qquad w\in\clp_z, 
\end{equation}
cf. \cite{lpp}, and by an appropriate scaling.
\section{Equilibria} \label{eq}
Observe that by construction the system (\ref{sos}) is a gradient flow w.r.t the functional
\[ \cle(\Omega_+,c) = \int_{C} c\ln c\,dx + |\Gamma| \] 
(cf. \cite{lpp} Section $2$ for a more detailed discussion of this fact). Hence, the functional $\cle$ is a Ljapunov function for the system (\ref{sos}). Indeed, assuming smoothness and strict positivity of concentrations, integration by parts yields  
\begin{equation} \label{lf2}
\frac{d}{dt} \; \cle(\Omega_+(t),c(t)) = - \int_C \frac{|\nabla c(t)|^2}{c(t)} - \int_{\Gamma(t)} (\lb c(t) \rb + H(t))^2 - \frac{1}{2} \int_C |\eps(u(t))|^2.
\end{equation}
Let $(u,q,c,\Omega_+)$ be an equilibrium solution to (\ref{sos}) (i.e. $(u,q,c,\Omega_+)$ is constant in time, $\Gamma$ is a closed connected hypersurface, and $u,q,c$ are continuously differentiable away from $\Gamma$). Since (\ref{lf2}) vanishes at equilibria, Korn's inequality implies that $u=0$. Moreover, $c$ must be constant in both phases and $\lb c \rb = -H$.  Thus, also $H$ is constant, so that $\Gamma$ is a sphere. The first equation in (\ref{sos}) implies 
then that $q$ is constant in both phases, and from the third equation one concludes that $\lb q \rb n = 0$ on $\Gamma$. Summarizing:
\begin{lemma}
A tupel $(u,q,c,\Omega_+)$ is an equilibrium solution to {\rm(\ref{sos})}  iff 
$\Omega_+$ is a ball of some radius $R$, $\bar{\Omega}_+ \subset C$, $\lb c \rb 
=  (N-1)/R$, $u = 0$ and $q$ is constant in $C$.
\end{lemma}
\section{Linearization at an equilibrium} \label{lineq}
We fix now a single equilibrium $(0,\tilde{q},\tilde{c},D_+)$ and  assume 
w.l.o.g. that $D_+ = \mathbb{B}(0,1)$ and $\lb \tilde c \rb = N-1 =: m$. We 
further define $S:=\partial \mathbb{B}(0,1)$, $D_- := C \setminus\bar 
D_+$ and keep 
these notations fixed hereafter. 

In order to solve system (\ref{sos}) we are going to consider a set of  
transformed equations given  over $D_\pm$ as fixed reference domains. The 
unknown family of surfaces $\{ \Gamma(t) \}$ will be described by a signed 
distance function with respect to the unit sphere. The ansatz is standard and 
has already been used in \cite{lipr2} in an identical way. 

The mapping 
\[
X: S \times (-1,1) \rightarrow \mathbb{R}^N, \qquad (x,s) \mapsto (1+s)x,
\]
is a smooth diffeomorphism onto its range. Fix $0<a<1$ small enough that  
$\bar{D} \subset C$, where $D:=\mbox{range}(X|_{S \times  (a,a)})$. As it is 
convenient, we decompose the inverse of $X:=X|_{S \times (-a,a)}$ into 
$X^{-1}=(P,\Lambda): D \rightarrow S \times (-a,a)$, where $P$ is the metric 
projection onto $S$, and $\Lambda$ is the signed distance function with respect 
to $S$, i.e.
$P(x)=x/|x|$, $\Lambda(x) = |x|-1$. Let $\tilde{a} \in (0,a/4)$ and
\[
\mbox{Ad}:=\{ \sigma \in C^1(S); \; \Vert \sigma \Vert_{C(S)} < \tilde{a} \}.
\]
It is well-known that, given $\sigma \in \mbox{Ad}$, the mapping  
$\theta_{\sigma}(x):=(1 + \sigma(x))x$ is a diffeomorphism mapping $S$ onto 
$S_\sigma:=\theta_\sigma[S]$. We extend this diffeomorphism to the whole of 
$\mathbb{R}^N$: Let $\chi \in C^{\infty}(\mathbb{R},[0,1])$ satisfy 
$\chi|_{[-\tilde{a},\tilde{a}]} \equiv 1$, $\chi|_{(-\infty,-3\tilde{a}]} \equiv 
\chi|_{[3\tilde{a},\infty)} \equiv 0$, $\Vert \chi' \Vert_{\infty} < 
1/\tilde{a}$. Then the mapping  
\begin{equation}
\label{Fort}
y \mapsto 
\left\{ \begin{array}{lcl}
X(P(y), \Lambda(y) + \chi(\Lambda(y)) \cdot \sigma(P(y)), & \mbox{if} & y \in D \\
y,                                         & \mbox{if} & y \not \in D \end{array} \right.
\end{equation}
$(\sigma \in \mbox{Ad})$, again denoted by $\theta_\sigma$, is an appropriate 
extension, the so-called Hanzawa diffeomorphism. We have 
$\theta_{\sigma} \in
\mbox{Diff}(\mathbb{R}^N,\mathbb{R}^N)$. Moreover, $\theta_{\sigma} \equiv 
\mbox{id}$ outside $D$,  in particular in a sufficiently small open neighborhood 
of $\partial C$. Moreover, denoting by $D_{\sigma,+}$ the domain enclosed by 
$S_\sigma$ and letting $D_{\sigma,-}:=C \setminus \bar{D}_{\sigma,+}$, we have 
that  
\[
\theta_\sigma |_{D_\pm} \in \mbox{Diff}(D_\pm,D_{\sigma,\pm}),
\]
$\sigma \in \mbox{Ad}$, and $\partial D_{\sigma,+} = S_\sigma$,  $\partial 
D_{\sigma,-} = S_\sigma \cup \partial C$. Finally note that the surface 
$S_\sigma$ is the zero level set of the function $\varphi_\sigma$ defined by
\[
\varphi_\sigma(x) = \Lambda(x) - \sigma(P(x)),
\]
$x \in D$, $\sigma \in \mbox{Ad}$, i.e. $S_\sigma = \varphi_\sigma^{-1}[\{ 0
\}]$. For later use we set
\[
L_\sigma(x) := |\nabla \varphi_\sigma|(\theta_\sigma(x)).
\]
It can be shown that $L_\sigma > 0$ on $S$ for all $\sigma \in \mbox{Ad}$. 
\newline
Given $\sigma \in \mbox{Ad}$, let $\theta^{*}_{\sigma}$, $\theta_{*}^{\sigma}$   
denote the pull-back and push-forward operators induced by $\theta_{\sigma}$, 
i.e. $\theta^{*}_{\sigma} \, f = f \circ \theta_{\sigma}$, $\theta_{*}^{\sigma} 
\, g = g \circ \theta_{\sigma}^{-1}$. If the functions $b,\rho$ are time 
dependent, i.e. $b=b(t,x)$, $\rho=\rho(t,x)$, we define $[\theta^{*}_{\rho} \, 
b] (t,x) := [\theta^{*}_{\rho(t)} \, b (t,\cdot)](x)$, analogue for 
$\theta_{*}^{\rho}$.

Using this notation, for $\rho: J \subset [0,\infty) \rightarrow \mbox{Ad} \cap 
C^2(S)$  and  sufficiently smooth $w_\pm \in \mathbb{R}^{D_\pm}$ we introduce 
the transformed operators
\[
\begin{array}{rcl}
n(\rho)  & := & \theta^{*}_{\rho} \, n_{[S_{\rho}]}; \\
H(\rho)  & := & \theta^{*}_{\rho} \, H_{[S_{\rho}]}; \\
\mathcal{A}_\pm(\rho)w_\pm & := & \theta^{*}_{\rho} (\Delta (\theta_{*}^{\rho} w_\pm)); \\
\mathcal{B}_\pm(\rho)w_\pm & := & \theta^{*}_{\rho} (\nabla (\theta_{*}^{\rho} 
w_\pm)  |_{S_{\rho}}) \cdot n(\rho); \\
\mathcal{K}_\pm(\rho)w_\pm & := & \theta^{*}_{\rho} (\nabla (\theta_{*}^{\rho} w_\pm)).
\end{array}
\]
Letting $\mu_\pm := c_\pm \circ \theta_\rho$, $\mu_{\pm,0} := c_{\pm,0} \circ  
\theta_{\rho_0}$,  instead of (\ref{sos}), we study the following problem on 
$D_\pm$ as fixed reference domains: 
\be \label{trp}
\left.\begin{array}{rcll}
\partial_t \mu_\pm - \kappa_\pm A(\rho) \mu_\pm + \mathcal{K}_\pm(\rho)  \mu_\pm 
\cdot s_\pm(\rho) + R_\pm(\rho,\mu) & = & 0 & \mbox{ in $D_\pm$,} \\
\kappa_\pm B(\rho) \mu_\pm + \mu_\pm (\lb \mu \rb + H(\rho)) & = & 0 &  \mbox{ 
on $\partial D_+$,} \\
\partial_n \mu_- & = & 0 & \mbox{ on $\partial C$,} \\
\partial_t \rho - L(\rho) [H(\rho) + \lb \mu \rb + s(\rho) \cdot n] & = & 0 &  \mbox{ on 
$\partial D_+$,} \\ 
\mu_\pm(0) & = & \mu_{\pm,0} & \mbox{ in $D_\pm$,} \\
\rho(0) & = & \rho_0, & 
\end{array}\right\}
\ee
where $s(\rho) := \theta^*_\rho u$,   
\be \label{stosolrho}
\left.\begin{array}{rclll}
-\nu^\pm \Delta u_\pm + \nabla q_\pm & = & 0 & \mbox{ in $D_{\rho,\pm}$,} & t > 0,\\
\vdiv u_\pm & = & 0 & \mbox{ in $D_{\rho,\pm}$,} & t > 0,\\
\lb \tau(u,q) \rb n & = & H n & \mbox{ on $S_\rho$,} & t > 0,\\ 
\lb u \rb & = & 0 &  \mbox{ on $S_\rho$,} & t > 0,\\
u_- & = & 0 & \mbox{ on $\partial C$,} & t > 0,
\end{array}\right\}
\ee
and $s_\pm(\rho) := s(\rho)|_{D_\pm}$. The terms 
$\clr_\pm$ arise  from the transformation of the time derivative $(\mu_\pm)_t$ 
and are determined by
\[
\clr_\pm(w_\pm,\sigma)(y) = r_0( L_\sigma [H(\sigma) + \lb w \rb + s(\rho)  
\cdot n(\rho)],B_{\mu}(\sigma)w_\pm)(y),  \qquad y \in D_\pm,
\]
where $w_\pm \in C^1(\overline{D}_\pm)$,  $\sigma \in \mbox{Ad}$ and
\begin{equation}
\label{BLM}
r_0(h,k)(y):=
\left\{ \begin{array}{lcl}
\chi(\Lambda(y)) \cdot h(P(y)) \cdot k(y), & \mbox{if} & y \in D \\
0,                                         & \mbox{if} & y \in \bar{C} \setminus 
D,  \end{array} \right.
\end{equation}
\[
B_{\mu}(\sigma)w_\pm (y) =  \theta^{*}_{\sigma} \, \nabla (\theta_{*}^{\sigma} 
w_\pm)  (y)  \cdot (n_{S} \circ P)(y) , \qquad y \in D_\pm
\]
($n_{S}$ being the exterior unit normal field of $S$).  The explicit calculation 

of $\clr_\pm$ is  straightforward, cf. again \cite{E04}. 

Linearization of (\ref{trp}) around the equilibrium  $(\mu_\pm, \rho) = 
(\tilde{c}_\pm, 0)$ yields the following system for the shifted variable $\mu - 
\tilde{c}$, denoted again by $\mu$: 
\be \label{lin1}
\left.\begin{array}{rcll}
\partial_t \mu_\pm - \kappa_\pm \Delta \mu_\pm & = & F_\pm(\mu_\pm, \rho) &  
\mbox{ in $D_\pm$,} \\
\kappa_\pm \partial_n \mu_\pm + \tilde{c}_\pm (\lb \mu \rb + (\Delta_S + m) 
\rho)  & = & G_\pm(\mu_\pm, \rho) & \mbox{ on $S$,} \\
\partial_n \mu_- & = & 0 & \mbox{ on $\partial C$,} \\
\partial_t \rho - [(\Delta_S + m) \rho + \lb \mu \rb + s'(0) \rho \cdot n]  & = 
& \tilde{H}(\mu_\pm, \rho) & \mbox{ on $S$,} \\
\mu_\pm(0) & = & \mu_{\pm,0} & \mbox{ in $D_\pm$,} \\
\rho(0) & = & \rho_0, & 
\end{array}\right\}
\ee
with suitable nonlinear remainders $F,G,H$ that act smoothly between the 
function spaces we are going to use, cf.  Lemma $4.2$ in \cite{lipr1} and 
Corollary \ref{trastoop} in the present paper. By construction, they satisfy 
\[F_\pm(0) = G_\pm(0) = \tilde{H}(0) = 0,\quad  F'_\pm(0) = 
G'_\pm(0) = \tilde{H}'(0) = 0.\] By $\Delta_S$ we denote the 
Laplace-Beltrami  operator of the unit sphere. After some algebra, letting 
$\alpha_\pm := \kappa_\pm / \tilde{c}_\pm$, $\tilde{\Delta} := \Delta_S + m$ and
\[ \begin{array}{rcl}
G_1(\mu_\pm,\rho) & = & G_+(\mu_\pm,\rho)/\tilde{c}_+; \\
G_2(\mu_\pm,\rho) & = & G_{+}(\mu_\pm,\rho)/\tilde{c}_+ -
G_{-}(\mu_\pm,\rho)/\tilde{c}_-; \\
G_3(\mu_\pm,\rho) & = & \tilde{H}(\mu_\pm,\rho) + G_1(\mu_\pm,\rho), 
\end{array} \]
we get 
\be \label{lin2}
\left.\begin{array}{rcll}
\partial_t \mu_\pm - \kappa_\pm \Delta \mu_\pm & = & F_\pm(\mu_\pm, \rho) &  
\mbox{ in $D_\pm$,} \\
\alpha_+ \partial_n \mu_+ + \lb \mu \rb + \tilde{\Delta} \rho & = & G_1(\mu_\pm, 
\rho) &  \mbox{ on $S$,} \\
\lb \alpha \partial_n \mu \rb & = & G_2(\mu_\pm, \rho) & \mbox{ on $S$,} \\
\partial_n \mu_- & = & 0 & \mbox{ on $\partial C$,} \\
\partial_t \rho + \alpha_+ \partial_n \mu_+ - s'(0) \rho \cdot n & = & 
G_3(\mu_\pm, \rho) &  \mbox{ on $S$,} \\
\mu_\pm(0) & = & \mu_{\pm,0} & \mbox{ in $D_\pm$,} \\
\rho(0) & = & \rho_0.  &  
\end{array}\right\}
\ee
We close this section by defining the abstract setting for our analysis 
and giving a precise statement of our main result. Let $p > N+2$. For $s \geq 0$ 
and a Banach space $Y$, $M \in \{ D_\pm, S, [0,T], [0,\infty) \}$ ($T > 
0$) we denote by $W_p^s(M,Y)$ the $L^p$-based Sobolev space of order $s$. In 
particular, if $s\notin\NNM$, this fractional-order Sobolev space coincides with 
the Besov space $B_{pp}^s(M,Y)$ (cf. \cite{Tr1}). For the sake of brevity we 
write $W_p^s(M):=W_p^s(M,\RRM)$ and introduce the notations $W_p^s(D_\pm) := 
W_p^s(D_+) \times W_p^s(D_-)$, $(\mu_\pm,\rho) := (\mu_+,\mu_-,\rho)$. Let  
\[
\begin{array}{rcl}
E_1 & := & \{ (\mu_\pm,\rho) \in W_p^2(D_\pm) \times W_p^{3-1/p}(S); \;  
\partial_n \mu_- = 0 \mbox{ on } \partial C\}; \\
E & := & \{ (\mu_\pm,\rho) \in W_p^{2-2/p}(D_\pm) \times W_p^{3-3/p}(S); \;  
\partial_n \mu_- = 0 \mbox{ on } \partial C\}; \\
E_0 & :=& L^p(D_\pm) \times W_p^{1-1/p}(S)
\end{array}
\]
and for an interval $J \subset [0,\infty)$ 
\[
\mathbb{E}(J) :=  L_p(J,E_1) \cap \big( W_p^{1}(J, L^p(D_\pm))  \times 
W_p^{(3-1/p)/2}(J,L_p(S)) \big). 
\] 
We further define spaces of exponentially decaying functions
\[
\mathbb{E}(\delta) := \{ (\xi_\pm,\sigma) \in \mathbb{E}(\mathbb{R}^+); \;  
e^{\delta t} (\xi_\pm,\sigma) \in \mathbb{E}(\mathbb{R}^+) \} 
\]
($\delta > 0$), equipped with the norm $\Vert (\xi_\pm,\sigma) \Vert_{\mathbb{E}(\delta)} := \Vert e^{\delta t} (\xi_\pm,\sigma) \Vert_{\mathbb{E}(\mathbb{R}^+)}$.
and recall the standard embedding result 
\begin{equation} \label{emb0}
\mathbb{E}(J) \hookrightarrow C(J,E).  
\end{equation}
We formally introduce the operators $\hat{L}$, $\hat{K}$  and $B$ 
by their action as follows:
\[
\hat{L}(\mu_\pm,\rho) := (\kappa_\pm \Delta \mu_\pm, - \alpha_+ \partial_n 
\mu_+),  \quad \hat{K}(\mu_\pm,\rho) := (0,0,s'(0) \rho|_S  \cdot n),
\]
\[
B(\mu_\pm,\rho) := (\alpha_+ \partial_n \mu_+ + \lb \mu \rb + \tilde{\Delta} 
\rho,  \lb \alpha \partial_n \mu \rb), 
\]
where $u = u(\rho) = s'(0) \rho$ (and a suitable $p$) solve 
\be \label{2pssphere}
\begin{array}{rcll}
\nu_\pm \Delta u_\pm - \nabla p_\pm & = & 0 &\mbox{ in $D_\pm$,}\\
\vdiv u_\pm & = & 0 & \mbox{ in $D_\pm$,}\\
\lb \tau(u,p) \rb n & = & \tilde{\Delta} \rho n & \mbox{ on $S$,} \\
\lb u \rb & = & 0 & \mbox{ on $S$,} \\
u_- & = & 0 & \mbox{ on $\partial C$}.
\end{array}
\ee
Then, with ${\bf F}:= (F_\pm,G_3)$, ${\bf G}:= (G_1,G_2)$, (\ref{lin2})   
can be written as an abstract evolution problem 
\begin{equation} \label{eveq}
\partial_t \mu - (\hat{L} + \hat{K}) (\mu) = {\bf F}(\mu), \quad B \mu =  {\bf 
G}(\mu),\quad \mu(0) = \mu_0 := (\mu_{\pm,0},\rho_0), \quad \mu := (\mu_\pm, 
\rho).
\end{equation}
Solutions to (\ref{eveq}) are paths in the manifold 
\[
\mathcal{M} := \{ \mu \in E; \; B \mu = {\bf G}(\mu) \}. 
\]
They are supposed to possess the following regularity:
\begin{definition} \label{Stsol}
A global strong solution of the evolution problem {\rm(\ref{eveq})} is a 
solution $\mu = (\mu_\pm,\rho): [0,\infty) \rightarrow E$ such that 
\[
\mu|_{[0,T]} \in \mathbb{E}([0,T]) \qquad \forall \; T > 0.
\]
\end{definition}
Observe that the set of equilibria of (\ref{eveq}) is 
\[
\mathcal{E} := \{ \eps \in \mathcal{M}; \; - (\hat{L} + \hat{K},B) (\eps) =  
({\bf F},{\bf G})(\eps) \},
\]
and that these equilibria correspond to the steady states of system (\ref{sos}). It is of crucial importance for our analysis that $\mathcal{E}$ is a  
submanifold of $\mathcal{M}$ of dimension $N+2$, cf. Lemma 2.1 in \cite{lipr2}, 
Proposition 6.4 in \cite{MSM}. Now we are prepared to state the main theorem of 
this paper:
\begin{theorem} \label{MTh}
There exist $\gamma, \delta > 0$ such that, given $\mu_0 \in 
\mathbb{B}_E(0,\gamma) \cap \mathcal{M}$, problem {\rm(\ref{eveq})} admits a 
unique global strong solution $\mu = \xi + e$, where  $(\xi,e) \in 
\mathbb{E}(\delta) \times \mathcal{E}$. Moreover, $\mu_0 \mapsto (\xi,e) \in 
C^1(\mathbb{B}_E(0,\gamma) \cap \mathcal{M}, \mathbb{E}(\delta) \times 
\mathcal{E})$.
\end{theorem}  
\section{Spectral analysis and proof of the main result} \label{sol}
In this section we study properties of the operator $L+K: D(L) \subset E_0  
\rightarrow E_0$, where $L \mu := \hat{L} \mu$, $K \mu := \hat{K} \mu$ and 
\[
\begin{array}{lcllcccc}
D(L) & := & \{ & (\mu_\pm, \rho) \in E_1; & & & \\
& & & \alpha_+ \partial_n \mu_+ + \lb \mu \rb + \tilde{\Delta} \rho & = & 0 &  
\mbox{ on $S$,} & \\
& & & \lb \alpha \partial_n \mu \rb & = & 0 & \mbox{ on $S$,} & \\
& & & \partial_n \mu_- & = & 0 & \mbox{ on $\partial C$.} & \} 
\end{array} 
\]
We will identify operators and vector spaces with their  complexifications 
without further mentioning. 
\begin{lemma} \label{spectrum}
\begin{itemize}
\item[\rm (i)] The spectrum of $L + K$ consists purely of isolated eigenvalues
having eigenspaces of finite dimension.
\item[\rm (ii)] The value $\lambda=0$ is an eigenvalue of $L + K$  with 
$\dim\cln(L + K)=N+2$. 
\item[\rm (iii)] All other eigenvalues of $L + K$ are real and negative.
\end{itemize}
\end{lemma}
\begin{proof} (i) The operator $L$ generates a strongly continuous analytic  
semigroup on $E_0$ (Theorem 2.2 in \cite{dpz}). As the operator $\rho \mapsto 
s'(0) \rho|_S \cdot n$ is of order $1$, $K$ is a relatively compact 
perturbation which implies that $L+K$ also generates a strongly continuous 
analytic semigroup on $E_0$ and in particular has a nonempty resolvent set. 
Since $D(L)$ is compactly embedded in $E_0$, the statement follows by \cite{Ka}, 
Theorem III. 6.29.

(ii) We introduce the bilinear form $\langle\cdot,\cdot\rangle$ by
\[\langle(w_+,w_-,\sigma),(u_+,u_-,\theta)\rangle  :=  \tilde 
c_-\int_{D_+}w_+u_+\,dx  + \tilde c_+\int_{D_-}w_-u_-\,dx 
-\tilde c_+\tilde c_-\int_S \sigma \tilde{\Delta} \theta \, dS.\]
Letting $u(\sigma) := s'(0)\sigma$ we observe that 
\begin{eqnarray} \label{Lnonneg}
& & -\langle (L + K) (w_+,w_-,\sigma),(\bar w_+,\bar 
w_-,\bar\sigma)\rangle\nonumber \\
&= &-\tilde c_-\kappa_+\int_{\Omega_+}\Delta w_+\bar w_+\,dx
-\tilde c_+\kappa_-\int_{\Omega_-}\Delta w_-\bar w_-\,dx
+\tilde c_+\tilde c_-\alpha_+\int_S\partial_n w_+(\alpha_+\partial_n \bar w_+
+\llbracket \bar w\rrbracket)\,dS\nonumber \\
  & &+ \int_{S} u(\sigma) \cdot n \tilde{\Delta} \bar\sigma \, dS\nonumber \\
&= &\tilde c_-\kappa_+\|\nabla w_+\|^2_{L^2(\Omega_+)}
+\tilde c_+\kappa_-\|\nabla w_-\|^2_{L^2(\Omega_-)}
+\kappa_+\alpha_+\tilde c_-\|\partial_n w_+\|^2_{L^2(S)}\nonumber \\
  && + \tfrac{1}{2} (\int_{D_+} |\eps(u(\sigma))|^2 \, \,dx + \int_{D_-}  
|\eps(u(\sigma))|^2 \,dx) \geq 0
\end{eqnarray}
for all $(w_+,w_-,\sigma)\in D(L)$. Suppose $(L+K)(w_+,w_-,\sigma)=0$ for  
$(w_+,w_-,\sigma)\in D(L)$. Then (\ref{Lnonneg}) implies that $w_\pm$ are 
constant on $D_\pm$ (thus $\llbracket w \rrbracket$ is constant on $S$) and that 
$u(\sigma) = 0$ in $C$. From elementary properties of $\Delta_S$ and the results 
from Sections \ref{stokeswsol}, \ref{stokesssol} we hence get that 
\be\label{basis}
\cln(L+K)={\rm span}\{(m,0,-1),\,(0,m,1),\,(0,0,x_1),\ldots,(0,0,x_N)\} =:  \{ 
\eps_1, ..., \eps_{N+2} \}. 
\ee

(iii) The computation (\ref{Lnonneg}) shows that  $-\langle (L + K) 
(w_+,w_-,\sigma),(\bar w_+,\bar w_-,\bar\sigma)\rangle \geq 0$ for all 
$(w_+,w_-,\sigma)\in D(L)$. Using the fact that $\int_S u(\sigma) \cdot n = 0$, 
we can show in completely the same fashion as in the proof of Lemma $3.1$ iii) 
in \cite{lipr2} that $\langle (w_+,w_-,\sigma),(\bar w_+,\bar 
w_-,\bar\sigma)\rangle \geq 0$ for all eigenvectors $(w_+,w_-,\sigma)$ of $L+K$ 
with equality only if $(w_+,w_-,\sigma) \in \mathcal{N}(L+K)$. Hence, the 
assertion follows from Lemma \ref{specnonneg} and (i).
\end{proof}

\begin{corollary}\label{gap}  We have
\[\sup{\rm Re}(\sigma(L+K)\setminus\{0\})<0.\]
\end{corollary}

\begin{lemma} \label{decomp} 
The zero eigenvalue of $L+K$ is semisimple, i.e. $X=\cln(L+K)\oplus \clr(L+K)$.
\end{lemma}
\begin{proof}
As $L + K$ has nonempty resolvent set and $D(L)$ is compactly  embedded in 
$E_0$, $L+K$ (considered as a bounded operator from $D(L)$ to $E_0$) is Fredholm 
and has index zero. Hence, by Corollary \ref{RN0c}, it suffices to show that 
$\cln(L+K) \cap \clr(L+K) = \{ 0 \}$. We introduce the linear mapping  
\[
\begin{array}{rcl}
\Phi(f_+,f_-,\theta) & := & ( \tilde{c}_+ \int_S \theta \,dS + \int_{D_+}  f_+ 
\,dx, \tilde{c}_- \int_S \theta \,dS - \int_{D_-} f_- \,dx) 
\end{array}
\]
and verify assumptions i), ii) and iii) of Lemma \ref{RN0a} with 
\[((w_+,w_-,\theta)|(v_+,v_-,\tau)) = \langle (w_+,w_-,\theta), 
(\bar{v}_+,\bar{v}_-, \bar{\tau}) \rangle\] and $V = \mathbb{C}^2$. Observe that 
by (\ref{Lnonneg}) (and the considerations below (\ref{Lnonneg})) we have that 
\[
\langle (L + K) (w_+,w_-,\sigma),(\bar w_+,\bar w_-,\bar\sigma) \rangle = 0  
\;\; \mbox{ iff } \;\; (\bar w_+,\bar w_-,\bar\sigma) \in \cln(L+K) 
\]
The divergence theorem and the fact that $\int_S u(\theta) \cdot n \, dS = 0$  
imply that $\Phi$ vanishes on $\clr(L+K)$. Assume that $z := (w_+,w_-,\sigma) 
\in \cln(L+K) \cap \cln(\Phi)$, $(w_+,w_-,\sigma) = \sum_{j = 1}^{N+2} \alpha_j 
\eps_j$ (cf. \ref{basis}). Then, as $\int_S x_j \, dS = 0$ (using $\alpha_\pm = 
\kappa_\pm / \tilde{c}_\pm$), $\Phi z = 0$ means that   
\[
\int_S (\alpha_2 - \alpha_1) \, dS = - \int_{D_+} \frac{\alpha_1 m}{\tilde{c}_+} \, dS 
\]
and  
\[
\int_S (\alpha_2 - \alpha_1) \, dS = + \int_{D_-} \frac{\alpha_2 m}{\tilde{c}_-} \, dS.
\]
This is equivalent to $A (\alpha_1,\alpha_2) = 0$, where 
\[
A=\left(\begin{array}{cc}
- m|D_+|/\tilde{c}_+ + |S| & - |S| \\ 
+|S|& m|D_-|/\tilde{c}_- - |S| 
\end{array}\right).
\]
We calculate, using $|S| = N|D_+|$, $\lb \tilde{c} \rb = m = N-1$
\[
\mbox{det}(A) = \frac{-m |S| |D_-|}{N \tilde{c}_-} + \frac{m 
|S||D_-|}{\tilde{c}_-}  + \frac{m|S||D_+|}{\tilde{c}_+} + \frac{m |D_+| 
|D_-|}{\tilde{c}_+} > 0, 
\]
i.e. $\alpha_1 = \alpha_2 = 0$. Hence, $z \in \mathcal{O}$.
\end{proof}
We turn to the proof of Theorem \ref{MTh}.  
Lemmas  \ref{spectrum} -- \ref{decomp} allow to follow the same 
strategy as in \cite{lipr2}, Section $4$. In fact, the arguments given there 
can literally be repeated here if one only  replaces the operators $\hat{L}$, 
$L$ by the operators $\hat{L} + \hat{K}$, $L + K$, respectively. However, the proof of 
\cite{lipr2}, Lemma 4.3 has to be modified  because Theorem 2.2 
in \cite{dpz} does not apply to the nonlocal operator $L + K$. Nevertheless, 
the analogous result holds true:

\begin{lemma} \label{isohp}
Let $\omega>0$. Then 
\[
(\omega-(\hat{L} + \hat{K}),B) \in \mathcal{L}_{\rm is}(E,W_p^{-2/p}(D_\pm) 
\times[W_p^{1-3/p}(S)]^3) \cap 
\mathcal{L}_{\rm is}(E_1,L_p(D_\pm)\times[W_p^{1-1/p}(S)]^3).
\]
\end{lemma}
In accordance with our general notation, we denote by $W_p^{-2/p}(D_\pm) 
:=B^{-2/p}_{pp}(D_\pm)$ a Besov space of negative
differentiability order, see \cite{Tr1}. 
\begin{proof}
From Lemma 4.3 in \cite{lipr2} we know that 
\begin{equation}\label{oldreg}
(\omega-\hat{L},B) \in 
\mathcal{L}_{\rm is}(E,W_p^{-2/p}(D_\pm)\times[W_p^{1-3/p}(S)]^3) \cap 
\mathcal{L}_{\rm is}(E_1,L_p(D_\pm)\times[W_p^{1-1/p}(S)]^3).
\end{equation}
In particular, $(\omega-\hat{L},B)$ is Fredholm and has index $0$. Since 
$\hat{K}$ is a  compact perturbation, the same is true for the operator 
$(\omega-(\hat{L} + \hat{K}),B)$. Therefore it suffices to show that 
$(\omega-(\hat{L} + \hat{K}),B)$ is injective. 

We first consider this operator as 
an element of $\mathcal{L}(E_1,L_p(D_\pm)\times[W_p^{1-1/p}(S)]^3)$. Then, 
injectivity is a direct consequence of Lemma 
\ref{spectrum}. 

To prove the remaining part, assume
\[(\omega-(\hat{L} + \hat{K}),B)\mu=0,\qquad\mu=(\mu_\pm,\rho)\in E,\]
or equivalently
\begin{equation}\label{homeq}
(\omega-\hat{L},B)\mu=(0,0,s'(0)\rho|_S\cdot n,0,0).
\end{equation}
Recall that $s'(0)\rho$ is defined by the BVP (\ref{2pssphere}).
Applying Theorem \ref{stoop2} to this problem and using $\rho\in 
W_p^{3-3/p}(S)$, $p>2$, we find that the right side of (\ref{homeq}) is in 
$L_p(D_\pm)\times[W_p^{1-1/p}(S)]^3$, so (\ref{oldreg}) yields $\mu\in E_1$. 
 The result follows again by Lemma \ref{spectrum}.
\end{proof}

\section{Conclusion}
Our analysis crucially relies on the fact that the problem under consideration 
belongs to the class of parabolic evolutions, in the general sense that the 
semigroup of operators (on appropriate function spaces) arising as solution 
of the linearized evolution problem is analytic. Corresponding maximal 
regularity results allow the treatment of the nonlinearities introduced by the 
transformation to a fixed domain. As typical for the techniques used here, they 
provide smooth solutions but are (in absence of further structural 
information) restricted to ``perturbative'' results, producing 
either short-time solutions (as in \cite{lpp} for the present problem) or 
long-time solutions near equilibria or periodic solutions.

The present paper shows that these techniques are strong and versatile enough 
to treat relatively complex models in which coupled evolutions in two phases 
and on their interface as well as additional elliptic systems occur. On a 
technical level, this is reflected in the fact that we use products of spaces 
of functions with different domains of definitions and a solution operator for 
the Stokes equations. In a sense, using this solution operator allows to 
treat the present problem as a perturbed version of the problem without 
flow, with the perturbation being ``of lower order.''

The convergence result may be viewed as an application of a suitably 
generalized principle of 
linearized stability to a nonlinear parabolic problem, which is also well 
established by now. Discussing the spectrum of the linearization at an 
equilibrium provides additional structural information to conclude that a 
solution starting close to the manifold of equilibria is actually global and 
converges to this manifold at an exponential rate.

In this respect, it remains an open and interesting question whether, and how, 
structural properties like parabolicity and stability of equilibria can be 
concluded already from properties of the initial ingredients of the variational 
model, and not only from the resulting moving boundary problem.

\section{Appendix}\label{app}
\subsection{Two-phase Stokes equations} \label{stokes}
Let $C \subset \mathbb{R}^N$ be the set defined in the introduction. In this 
section  we denote by $\Omega_+$ a bounded simply connected open set with 
boundary $\partial \Omega_+$ of class $C^\infty$ such that $\bar{\Omega}_+ 
\subset C$ and define $\Omega_- := C \setminus \bar{\Omega}_+$. Moreover, $n$ 
denotes the outward unit normal field of $\partial \Omega_+$. If no confusion 
seems likely, the symbol $\partial_n$ stands for both the directional derivative 
w.r.t. $n$ and w.r.t. the outer unit normal field of $\partial C$. We are 
interested in the two-phase Stokes system 
\be \label{2psC1}
\begin{array}{rcll}
- \nu_\pm \Delta u_\pm + \nabla p_\pm & = & f_\pm &\mbox{ in $\Omega_\pm$,}\\
- \vdiv u_\pm & = & g_\pm & \mbox{ in $\Omega_\pm$,}\\
\lb \tau(u,p) \rb n & = & h & \mbox{ on $\partial \Omega_+$,} \\
\lb u \rb & = & l & \mbox{ on $\partial \Omega_+$,} \\
u_- & = & 0 & \mbox{ on $\partial C$}
\end{array}
\ee
and consider first the question of unique solvability of the simplified problem
\be \label{2psb}
\begin{array}{rcll}
- \nu_\pm \Delta u_\pm + \nabla p_\pm & = & f_\pm &\mbox{ in $\Omega_\pm$,}\\
- \vdiv u_\pm & = & g_\pm & \mbox{ in $\Omega_\pm$,}\\
\lb \tau(u,p) \rb n & = & h & \mbox{ on $\partial \Omega_+$,} \\
\lb u \rb & = & 0 & \mbox{ on $\partial \Omega_+$,} \\
u_- & = & 0 & \mbox{ on $\partial C$}.
\end{array}
\ee 
\subsubsection{Weak solutions} \label{stokeswsol}
Let $H := H_0^1(C,\mathbb{R}^N) = \{ w \in W_2^1(C,\mathbb{R}^N);\; w = 0 \mbox{ on } \partial C\}$, $Q := L_2(C)$, 
\[
a := H \times H \rightarrow \mathbb{R}, \quad (u,\varphi) \mapsto \int_{C} \nu (\eps(u):\eps(\varphi)), 
\]
$\nu = \nu_\pm$ in $\Omega_\pm$, and
\[
b := Q \times H \rightarrow \mathbb{R}, \quad (q,\varphi) \mapsto - \int_{C} q \mbox{div} \varphi.
\]
A weak solution of the system \ref{2psb} ($(f,g): C \rightarrow \mathbb{R}^{N+1}$, $h: \partial \Omega_+ \rightarrow \mathbb{R}^{N}$) is a pair $(u,[q]) \in H \times Q/\sim_{c}$ ($f \sim_{c} g :\Leftrightarrow f = g + \mbox{const}$) that satisfies 
\[
a(u,\varphi) + b(q,\varphi) = \int_{\partial \Omega_+} h \varphi + \int_{C} f \varphi + \int_{C} \nabla g \varphi \quad \mbox{ for all } \varphi \in H \quad (q \in [q])
\]
as well as 
\[
b(\psi,u) = g \quad \mbox{ for all } \psi \in Q.
\]
Due to Korn's inequality, the bilinear form $a$ is coercive on $H$. Moreover, the bilinear form $b$ induces a linear operator 
\[
B: H \rightarrow Q', \quad Bu(\psi) := b(\psi,u).
\]
Identifying $Q$ with its dual by means of the Riesz isomorphism the range of $B$ is the set $\{ r \in L_2(C);\; \int_C r = 0 \}$. Since this is a closed subset of $L_2(C)$, classical results (cf. \cite{BF91}, Section II.1) imply that there is a unique weak solution of (\ref{2psb}) for every $(f,g,h) \in L^{2}(C, \mathbb{R}^N) \times H^{1}(C, \mathbb{R}) \times L_2(\partial \Omega_+, \mathbb{R}^N)$ provided $\int_C g = 0$. 
\subsubsection{The Lopatinskii-Shapiro condition} \label{ls}
We want to show that the two-phase Stokes system (\ref{2psC1}) satisfies the Lopatinskii-Shapiro condition. W.l.o.g. we restrict ourselves to the halfspace situation, i.e. the case $\Omega_\pm := \mathbb{R}^{N-1} \times \mathbb{R}^{\pm}$. 

Reflecting $u_-$ in system (\ref{2psC1}) to the upper half space, the operator on the left hand side of system (\ref{2psC1}) can be expressed by the $(2N+2) \times (2N+2)$ and the $(2N) \times (2N+2)$ matrices of operators 
\[
A(\partial_1,...,\partial_N) := \left(\begin{smallmatrix} -\nu_+ {\bf \Delta} & \nabla & 0 & 0 \\ \nabla^T & 0 & 0 & 0 \\ 0 & 0 & -\nu_- {\bf \Delta} & \tilde{\nabla} \\ 0 & 0 & \tilde{\nabla}^T & 0 \end{smallmatrix}\right), \quad B(\partial_1,...,\partial_N) := \left(\begin{smallmatrix} {\bf B^+} & - {\bf e_N} & {\bf -B^-} & {\bf e_N} \\ {\bf I} & {\bf 0 }& {\bf -I} & {\bf 0} \end{smallmatrix}\right)
\]
both acting on vectors $(u_+,p_+,u_-,p_-)^T$ ($u_\pm := (u^1_\pm, ..., u^N_\pm)$). Here, we used the $N \times N$-matrices of operators 
\[
{\bf \Delta} := \left(\begin{smallmatrix} \Delta &  & \mbox{\large{0}} \\ & \ddots &  \\ \mbox{\large{0}} &  & \Delta \end{smallmatrix}\right), \quad {\bf B^\pm} := \nu_\pm \; \left(\begin{smallmatrix} \pm \partial_N &  & \mbox{\large{0}} & \partial_1 \\ & \ddots &  &  \vdots \\ \mbox{\large{0}} &  & \pm \partial_N & \partial_{N-1} \\  &  &  & \pm 2 \partial_N \end{smallmatrix}\right)
\]
and the notation $\tilde{\nabla} := (\partial_1,\hdots,\partial_{N-1},-\partial_N)^T$, ${\bf e_N} := (0,...,0,1) \in \mathbb{R}^N$. The operator $A$ represents a Douglis-Nirenberg elliptic system (cf. \cite{WRL}) with DN-numbers 
\[
s_1 = ... = s_N = s_{N+2} = ... = s_{2N+1} = t_1 = ... = t_N = t_{N+2} = ... = t_{2N+1} = 1, 
\]
\[
s_{N+1} = s_{2N+2} = t_{N+1} = t_{2N+2} = 0
\]
and it coincides with its principal part (note that $\sum s_j + t_j = 4N = 
\mbox{ord}(A)$). The characteristic polynomial is $\nu_+^{N-1} \nu_-^{N-1} 
(|\xi|^2 + \lambda^2)^{2N}$, where $\xi = (\xi_1,...,\xi_{N-1}) \in 
\mathbb{R}^{N-1}$ and $\lambda \in \mathbb{R}$. 

We have to determine a $2N$-dimensional space $\mathcal{M}^0$ of exponentially decaying solutions to the initial value problem $A(i\xi_1,...,i\xi_{N-1},\partial_t)(u_+,p_+,u_-,p_-) = 0$, i.e. (letting $v_\pm := u^N_\pm$)
\begin{equation} \label{ivp}
\left\{\begin{array}{rcll}
\nu_\pm (|\xi|^2 - \partial_t^2) u_\pm^j, & = &   -i \xi_j p_\pm & \quad j = 1,...,N-1, \\
\nu_\pm (|\xi|^2 - \partial_t^2) v_\pm, & = & \mp \partial p_\pm & \\
\sum_{j=1}^{N-1} \xi_j u_\pm^j & = & \pm i \partial v_\pm.
\end{array}\right.
\end{equation}
This system can be solved to the result  
\begin{equation} \label{sol1}
\left\{\begin{array}{rclll}
u_\pm^j & = & \alpha^j_\pm e^{-|\xi| t} & j = 1,...,N-1, & \alpha^j_\pm \in \mathbb{R}, \\
v_\pm   & = & \frac{\pm i}{|\xi|} (\alpha_\pm|\xi) e^{-|\xi| t}, & \alpha_\pm := (\alpha_\pm^1,...,\alpha_\pm^{N-1}) \\
p_\pm   & = & 0    
\end{array}\right.
\end{equation}
and  
\begin{equation} \label{sol2}
\left\{\begin{array}{rcll}
\tilde{u}_\pm^j & = & \beta_\pm i (\frac{\xi_j}{|\xi|^2} - \frac{t \xi_j}{|\xi|}) e^{-|\xi| t}, & j = 1,...,N-1, \\
\tilde{v}_\pm   & = & \pm \beta_\pm t e^{-|\xi| t}, \\
\tilde{p}_\pm   & = & 2 \beta_\pm \nu_\pm e^{-|\xi| t},   
\end{array}\right.
\end{equation}
%which defines a basis of $\mathcal{M}^0$ by representing the vector $(\alpha_+,\alpha_-)$ w.r.t any basis of $\mathbb{R}^{2N-2}$. 
$\beta_\pm \in \mathbb{R}$. Next we show that the problem
\[ 
B(i \xi_1,...,i \xi_{N-1},\partial_t)(u_+,p_+,u_-,p_-)|_{t=0} = 0
\]
possesses in $\mathcal{M}^0$ only the trivial solution. Writing simply $(u^1_+,...,u^{N-1}_+,v_+)$ instead of $(u^1_\pm,...,u^{N-1}_\pm,v_\pm) + (\tilde{u}^1_\pm,...,\tilde{u}^{N-1}_\pm,\tilde{v}_\pm)$, we assume that  
\begin{equation} \label{cont}
\Big[ (u^1_+,...,u^{N-1}_+,v_+) - (u^1_-,...,u^{N-1}_-,v_-) \Big] (0) = 0.
\end{equation}
This implies that
\begin{equation} \label{rel0}
\alpha^j_+ - \alpha^j_- + \frac{i \xi_j}{|\xi|^2} (\beta_+ - \beta_-) = 0, \;\; (j = 1,...,N-1), \quad (\alpha_+|\xi) = -(\alpha_-|\xi).
\end{equation}
Multiplication with $\xi$ yields 
\begin{equation} \label{rel1}
(\alpha_\pm|\xi) = \frac{\mp i}{2} (\beta_+ - \beta_-).
\end{equation}
If, additionally, 
\begin{equation} \label{stress}
\left[ \nu_+ \left [\left(\begin{smallmatrix} \partial_t u_+^1 \\ \vdots  \\ \partial_t u_+^{N-1} \\ \partial_t v_+ \end{smallmatrix}\right) + \left(\begin{smallmatrix} i \xi_1 v_+ \\ \vdots  \\ i \xi_{N-1} v_+ \\ \partial_t v_+ \end{smallmatrix}\right) \right] - \nu_- \left[ \left(\begin{smallmatrix} - \partial_t u_-^1 \\ \vdots  \\ - \partial_t u_-^{N-1} \\ - \partial_t v_- \end{smallmatrix}\right) + \left(\begin{smallmatrix} i \xi_1 v_- \\ \vdots  \\ i \xi_{N-1} v_- \\ - \partial_t v_- \end{smallmatrix}\right) \right] - \left(\begin{smallmatrix} 0 \\ \vdots  \\ 0 \\ p_+ - p_- \end{smallmatrix}\right) \right] (0) = 0,
\end{equation}
we find that 
\begin{equation} \label{rel2}
\begin{array}{ll}
  & \nu_+ \left[ \left(\begin{smallmatrix} - \alpha_+ |\xi| \\ - i (\alpha_+|\xi) \end{smallmatrix}\right) + \left(\begin{smallmatrix} - \xi (\alpha_+|\xi) / |\xi|  \\ - i (\alpha_+ |\xi) \end{smallmatrix}\right) + \left(\begin{smallmatrix} -2 i \beta_+ \xi / |\xi| \\  \beta_+ \end{smallmatrix}\right) + \left(\begin{smallmatrix} 0 \\ 	 \beta_+ \end{smallmatrix}\right) \right] \\
  & \\
- & \nu_- \left[ \left(\begin{smallmatrix} \alpha_- |\xi| \\ - i (\alpha_-|\xi) \end{smallmatrix}\right) + \left(\begin{smallmatrix}  \xi (\alpha_-|\xi) / |\xi|  \\ - i (\alpha_- |\xi) \end{smallmatrix}\right) + \left(\begin{smallmatrix} 2 i \beta_- \xi / |\xi| \\  \beta_- \end{smallmatrix}\right) + \left(\begin{smallmatrix} 0 \\  \beta_- \end{smallmatrix}\right) \right] \\
  & \\
- & \left[ \nu_+ \left(\begin{smallmatrix} 0 \\ 2 \beta_+ \end{smallmatrix}\right) - \nu_- \left(\begin{smallmatrix} 0  \\ 2 \beta_- \end{smallmatrix}\right) \right] \\
  & \\
 & = 0.
\end{array}
\end{equation}
By multiplying the first line in (\ref{rel2}) with $\xi$ and by using (\ref{rel1}) this gives the linear system 
\begin{equation} \label{beta}
M \vec{\beta} := \left(\begin{smallmatrix} \nu_+ + \nu_-  & \nu_+ + \nu_- \\ \nu_+ + \nu_- & - (\nu_+ + \nu_-) \end{smallmatrix}\right)   \left(\begin{smallmatrix} \beta_+ \\ \beta_- \end{smallmatrix}\right) = 0,
\end{equation}
which possesses only the trivial solution $\beta_+ = \beta_- = 0$ since $\mbox{det}(M) = -2 (\nu_+ + \nu_-)^2 < 0$. Hence, by (\ref{rel0}), (\ref{rel1}) $\alpha_+ = \alpha_-$ and $(\alpha_+|\xi) = (\alpha_-|\xi) = 0$. Using this, the first line in (\ref{rel2}) reduces to  
\[
(\nu_+ + \nu_-) |\xi| \alpha_+ = 0
\]
and hence also $\alpha_+ = \alpha_- = 0$. Therefore, the Lopatinskii-Shapiro condition is satisfied.
\subsubsection{Regularity} \label{stokesssol}
We are now interested in strong/classical solutions of the system (\ref{2psC1}) under the necessary solvability demand 
\begin{equation} \label{comp}
\int_{\Omega_+} g_+ + \int_{\Omega_-} g_- = - \int_{\partial \Omega_+} l \cdot n. 
\end{equation}
Let  
\begin{equation} \label{stoop}
\begin{array}{l}
\Lambda_{\nu_+,\nu_-} (u_+,p_+,u_-,p_-) := \\
(\nu_+ \Delta u_+ - \nabla p_+, \nu_- \Delta u_- - \nabla p_-, - \mbox{div} u_+, - \mbox{div} u_-, \lb \tau(u,p) \rb n , \lb u \rb). 
\end{array}  
\end{equation}
From Theorem $9.32$ in \cite{WRL} and Section \ref{ls} we know that the operators $\Lambda_{\nu_+,\nu_-}$, considered as bounded operator between appropriate function spaces (see Theorem \ref{stoop2} below), are Fredholm for all positive $\nu_+, \nu_-$. In order  to calculate their index, we first consider the case that $\nu_+ = \nu_- =: \nu > 0$ and determine the range of $\Lambda_{\nu,\nu}$ (for the sake of brevity we refrain from stating regularities as they can be easily added by means of classical elliptic theory and the results from \cite{L69}, Section $3.3$, $3.5$): let $l_+, l_-$ satisfy
\begin{itemize}
\item $l_+ - l_- = l$ on $\partial \Omega_+$;
\item $\int_{\Omega_\pm} g_\pm = \mp \int_{\partial \Omega_+} l_\pm \cdot n$.
\end{itemize}
Taking into account (\ref{comp}), one possible choice is $l_+ := l + l_-|_{\partial \Omega_+}$, where $l_- := \nabla L$ and $L$ solves
\be \label{contbd}
\begin{array}{rcll}
- \Delta L & = & g_- &\mbox{ in $\Omega_-$,}\\
\partial_n L & = & \frac{\int_{\Omega_-} g_-}{|\partial \Omega_+|} & \mbox{ on $\partial \Omega_+$,}\\
\partial_n L & = & 0 & \mbox{ on $\partial C$.}\\
\end{array}
\ee
Further, let $w_\pm := \nabla W_\pm$, where $W_\pm$ solve
\be \label{divI}
\begin{array}{rcll}
- \Delta W_+ & = & g_+ &\mbox{ in $\Omega_+$,}\\
\partial_n W_+ & = & l_+ \cdot n & \mbox{ on $\partial \Omega_+$,}\\
\end{array}
\ee
\be \label{divE}
\begin{array}{rcll}
- \Delta W_- & = & g_- &\mbox{ in $\Omega_-$,}\\
\partial_n W_- & = & l_- \cdot n & \mbox{ on $ \partial \Omega_+$,}\\
\partial_n W_- & = & 0 & \mbox{ on $\partial C$.}\\
\end{array}
\ee
Note that $- \mbox{div} w_\pm = g_\pm$ in $\Omega_\pm$. We extend $f_-$ and $\Delta w_-$ to $\mathbb{R}^N \setminus \bar{\Omega}_+$ in such a way that they vanish outside some open ball containing $\bar{C}$ and consider the problems 
\be \label{2psCI}
\begin{array}{rcll}
- \nu \Delta v_+ + \nabla q_+ & = & f_+ + \nu \Delta w_+ &\mbox{ in $\Omega_+$,}\\
\vdiv v_+ & = & 0 & \mbox{ in $\Omega_+$,}\\
v_+ & = & l_+ - w_+ & \mbox{ on $\partial \Omega_+$,}
\end{array}
\ee
\be \label{2psCE}
\begin{array}{rcll}
- \nu \Delta v_- + \nabla q_- & = & f_- + \nu \Delta w_- &\mbox{ in $\mathbb{R}^N \setminus \bar{\Omega}_+$,}\\
\vdiv v_- & = & 0 & \mbox{ in $\mathbb{R}^N \setminus \bar{\Omega}_+$,}\\
v_- & = & l_- - w_- & \mbox{ on $\partial \Omega_+$.} \\
\end{array}
\ee
It follows from Section 3.5 in \cite{L69} that the problems (\ref{2psCI}) and (\ref{2psCE}) possess classical solutions (since $w_\pm \cdot n = l_\pm \cdot n$ on $\partial \Omega_\pm$). Moreover, $\int_{\partial C} v_- \cdot n_{\partial C} = 0$. Next we are interested in the system
\be \label{2psC2}
\begin{array}{rcll}
- \nu \Delta u_+ + \nabla p_+ & = & 0 &\mbox{ in $\Omega_+$,}\\
- \nu \Delta u_- + \nabla p_- & = & 0 &\mbox{ in $\mathbb{R}^N \setminus \bar{\Omega}_+$,}\\
\vdiv u_+ & = & 0 & \mbox{ in $\Omega_+$,}\\
\vdiv u_- & = & 0 & \mbox{ in $\mathbb{R}^N \setminus \bar{\Omega}_+$,}\\
\lb \tau(u,p) \rb n & = & h - \lb \tau(w+v,q) \rb n& \mbox{ on $\partial \Omega_+$,} \\
\lb u \rb & = & 0 & \mbox{ on $\partial \Omega_+$,} 
\end{array}
\ee
where $v,w: \Omega_+ \cup \Omega_- \rightarrow \mathbb{R}^{N}$, $q: \Omega_+ \cup \Omega_- \rightarrow \mathbb{R}$ are defined in the obvious way. The single layer potential with density $\psi$ and w.r.t. the constant viscosity $\nu > 0$ is given by  
\begin{equation} \label{slp}
\begin{array}{rcl}
V(x,\psi) & := & \frac{1}{2 \nu \omega_N} \int_\Gamma ( \frac{1}{(n-2)|x-y|^{N-2}} + \frac{(x-y)(x-y)^T}{|x-y|^N} ) \psi(y) \;  d\sigma(y); \\
Q(x,\psi) & := & \frac{1}{\omega_N} \int_\Gamma ( \frac{(x-y)}{|x-y|^N} ) \psi(y) \;  d\sigma(y). \\
\end{array}
\end{equation}
As it can be seen from the results in \cite{L69}, Chapter 3, the restrictions $(u_\pm, p_\pm)$ of 
\[
(V(\cdot,h - \lb \tau(w+v,q) \rb n), Q(\cdot,h - \lb \tau(w+v,q) \rb n))
\]
to $\Omega_+$ and $\mathbb{R}^N \setminus \bar{\Omega}_+$, respectively, solve (\ref{2psC2}) in a classical sense, provided $h - \lb \tau(w+v,q) \rb n$ is  continuous (observe that precise regularity properties of $(u_\pm,p_\pm)$ can be obtained from the fact that the mapping $\psi \mapsto V(\cdot,\psi)|_{\partial \Omega_+}$ is a pseudodifferential operator of order $-1$ as well as regularity theory for the Stokes-Dirichlet problem, cf. Section $3.3$, $3.5$ in \cite{L69}). 

Since $u_+$ is divergence free, it follows that $\int_{\partial \Omega_+} u_+ \cdot n = 0$, and $\lb u \rb = 0$ on $\partial \Omega_+$ implies $\int_{\partial \Omega_+} u_- \cdot n = 0$. Hence, since also $u_-$ is divergence free, it follows that $\int_{\partial C} u_- \cdot n_{\partial C} = 0$. Thus, $\int_{\partial C} (u_- + v_-) \cdot n_{\partial C} = 0$.

Let $(\Phi,P)$ a (smooth across $\partial \Omega_+$) solution of the following Dirichlet problem for the Stokes equations (which exists since $\int_{\partial C} (u_- + v_-) \cdot n_{\partial C} = 0$, cf. \cite{L69}, Chapter 3): 
\be \label{Dir}
\begin{array}{rcll}
- \nu \Delta \Phi + \nabla P & = 0 &  &\mbox{ in $C$,}\\
\vdiv \Phi & = & 0 & \mbox{ in $C$,}\\
\Phi & = & - (u_- + v_- + w_-) & \mbox{ on $\partial C$.}
\end{array}
\ee
Summarizing, the pair 
\[
(w_\pm + v_\pm + u_\pm + \Phi, q_\pm + p_\pm + P) 
\]
is easily seen to solve (\ref{2psC1}) (with $\nu_+ = \nu_- = \nu$) in  a 
classical sense. Therefore, the necessary solvability demand  (\ref{comp}) is 
also sufficient. Hence, the range of $\Lambda_{\nu,\nu}$ is of codimension $1$. 
Since we know from Theorem \ref{stokeswsol} that the kernel of 
$\Lambda_{\nu,\nu}$ is one dimensional, this operator has index $0$. 
Consequently ($\nu > 0$ was arbitrary), by homotopic stability of the index, all 
members of the family $\{ \Lambda_{\nu_+, (1-t) \nu_+ + t \nu_-};\; t \in [0,1] 
\}$ have index $0$, in particular $\Lambda_{\nu_+,\nu_-}$. Since also this 
operator has a one dimensional kernel, the following theorem \ref{stoop2} can be 
deduced from the general theory of elliptic boundary value problems (cf. Section 
$4$ in \cite{Tr1}, Theorem $9.32$ in \cite{WRL}). In order to economize notation 
we introduce the quotient spaces $\tilde{\mathcal{F}} := \mathcal{F}/\sim_{c}$, 
where $\mathcal{F} \in \{ W^\alpha_p, C^\alpha, c^\alpha \}$ and $\sim_{c}$ is 
the equivalence relation introduced in Section \ref{stokeswsol}.  Here, 
$C^{\alpha}$ stands for the usual H\"older space and $c^{\alpha}$ denotes the 
little H\"older space, that is the closure of the smooth functions in 
$C^{\alpha}$.
\begin{theorem} \label{stoop2}
Let $(r,\beta,p,k) \in (0,1] \times (0,1) \times [1,\infty) \times  (\mathbb{N} 
\cup \{ 0 \})$ satisfy $k+r > 1/p$, $r \neq 1/p$ and let  $\mathcal{C} \in \{ c, 
C \}$. Suppose that $(f_\pm,g_\pm,h,l)$ of class 
\begin{itemize}
\item[i)] $W := [W^{k+r-1}_p(\Omega_\pm)]^N \times W^{k+r}_p(\Omega_\pm)  \times 
[W^{k+r-1/p}_p(\partial \Omega_+)]^N \times [W^{k+r+1-1/p}_p(\partial 
\Omega_+)]^N$ 
\item[] or  
\item[ii)] $\mathcal{C} := [\mathcal{C}^{k+\beta}(\bar{\Omega}_\pm)]^N \times  
\mathcal{C}^{k+1+\beta}(\bar{\Omega}_\pm) \times 
[\mathcal{C}^{k+1+\beta}(\partial \Omega_+)]^N \times 
[\mathcal{C}^{k+2+\beta}(\partial \Omega_+)]^N$
\end{itemize}
satisfies $\int_{\Omega_+} g_+ + \int_{\Omega_-} g_- = - \int_{\partial 
\Omega_+} l  \cdot n$.
In both cases problem  {\rm(\ref{2psC1})} possesses a solution $(u_\pm,p_\pm)$ 
which is unique up to an additive constant for $p_\pm$. This constant is the 
same in both phases. In case i) $(u_\pm,p_\pm)$ belongs to the class 
$[W^{k+r+1}_p(\Omega_\pm)]^N \times W^{k+r}_p(\Omega_\pm)$ and satisfies the a 
priori estimate 
\begin{equation} \label{apr1}
\Vert (u_\pm,[p_\pm]) \Vert_{[W^{k+r+1}_p(\Omega_\pm)]^N  \times 
\tilde{W}^{k+r}_p(\Omega_\pm)} \leq \gamma \Vert (f_\pm,g_\pm,h,l) \Vert_{W}
\end{equation}
with a positive constant $\gamma$ independent of $(f_\pm,g_\pm,h,l)$.  In case 
ii) $(u_\pm,p_\pm)$ belongs to the class 
$[\mathcal{C}^{k+2+\beta}(\bar{\Omega}_\pm)]^N \times 
\mathcal{C}^{k+1+\beta}(\bar{\Omega}_\pm)$ and satisfies the estimate 
\begin{equation} \label{apr2}
\Vert (u_\pm,[p_\pm]) \Vert_{[\mathcal{C}^{k+2+\beta}(\bar{\Omega}_\pm)]^N 
\times  \tilde{\mathcal{C}}^{k+1+\beta}(\bar{\Omega}_\pm)} \leq \gamma \Vert 
(f_\pm,g_\pm,h,l) \Vert_{\mathcal{C}}
\end{equation}
with a positive constant $\gamma$ independent of $(f_\pm,g_\pm,h,l)$. The  
symbol $[p_\pm]$ stands for the class $\{ p_\pm + \zeta;\; \zeta \in \mathbb{R} 
\}$.
\end{theorem}
For a given domain $\bar{\Omega}_+ \subset C$ with smooth boundary we denote by  
$\Lambda_{\Omega_\pm}$ the operator defined in (\ref{stoop}). Using this 
notation (and those defined in Section \ref{lineq}), we are able to work out 
that the solution operator $s(\rho)$ of the problem (\ref{stosolrho}) depends 
smoothly on $\rho$. 
\begin{corollary} \label{trastoop}
There is a neighbourhood $U$ of $0$ in $W_p^{3-3/p}(S)$ such that $s(\rho)$ is 
well  defined for all $\rho \in U \cap \mbox{Ad}$ and $s \in C^{\infty}(U \cap 
\mbox{Ad},[W^{2-2/p}_p(D_\pm)]^N)$. Moreover, if $\delta \geq 0$ is given, then 
there is a neighbourhood $V$ of $0$ in $\mathbb{E}(\delta)$ such that 
$\tilde{s}$ given by   
\[
\tilde{s}((\mu_\pm,\sigma))(t) := (0,0,s(\sigma(t))|_S \cdot n_S)    
\]
is well defined for all $(\mu_\pm, \sigma) \in V$ and  $\tilde{s}  \in 
C^{\infty}(V,L_p(\mathbb{R}^+,E_0))$. 
\end{corollary}
\begin{proof}
Let $\Lambda(\rho) := \theta^{*}_\rho \Lambda_{(\Omega_\rho)_\pm} 
\theta_{*}^\rho$.  Carrying out the transformation of the differential operators 
involved in (\ref{stosolrho}) (cf. the proofs of Theorem $3.1$ in \cite{lpp}, 
Lemma $4.2$ in \cite{lipr1}) we obtain 
\[
s(\rho) = P \Lambda(\rho)^{-1}(0,0,H(\rho)n(\rho),0),
\]
where $P(u_+,p_+,u_-,p_-) := u$ and $\Lambda$, $H$ and $n$ depend all smoothly 
on $\rho$. Since $\Lambda(0)$ is precisely the operator considered in Theorem  
\ref{stoop2} (with $\Omega_\pm := D_\pm$), the first assertion is obtained by  
standard perturbation arguments for isomorphisms. 

Observe that due to the embedding (\ref{emb0}) we may assume  that 
$\sigma[\mathbb{R}^+] \subset \mbox{Ad} \cap W_p^{3-3/p}(S)$ provided $\Vert 
(\mu_\pm,\sigma) \Vert_{\mathbb{E}(\delta)}$ is small enough. Hence the second 
assertion is a consequence of mapping and  smoothness properties of concerning 
Nemytskij operators, cf. \cite{RuSi}. 
\end{proof}

\subsection{A few functional analytic tools}
\begin{lemma} \label{RN0a}
Let $H,V$ be complex vector spaces and $A: D(A) \subset H \rightarrow H$ be a linear operator. Let $(\cdot|\cdot)$ be a sesquilinear form on $D(A) \times D(A)$, $\Phi \in \mathcal{L}(H,V)$ and  
\[
\begin{array}{rcll}
\mathcal{O} & := & \{ z \in D(A); \; (z|h) = 0 \mbox{ for all } h \in D(A) \}.
\end{array}
\]
Assume that 
\begin{itemize}
\item[i)]   $(Au|u) = 0$ iff $u \in \cln(A)$;
\item[ii)]  $\clr(A) \subset \cln(\Phi)$;
\item[iii)] $\cln(\Phi) \cap \cln(A) \subset \mathcal{O}$;
\end{itemize}
Then $\clr(A) \cap \cln(A) = \{ 0 \}$. 
\end{lemma}
\begin{proof}
Let $z \in \clr(A) \cap \cln(A)$, $z = Ay$ for some $y \in D(A)$. Then ( because of ii), iii) ) $z \in \mathcal{O}$ and 
\[
0 = (z|y) = (Ay|y), 
\]
and hence, because of i), $y \in \cln(A)$ i.e. $z = Ay = 0$. 
\end{proof}
\begin{lemma} \label{specnonneg}
Let $H$ be a complex vector space, $A: D(A) \subset H \rightarrow H$ and let $(\cdot|\cdot)$ be a sesquilinear  form on $D(A) \times D(A)$. Assume additionally that  
\begin{itemize}
\item[i)]  $(Au|u) \in [0,\infty)$ for all $u \in D(A)$;
\item[ii)]  $(u|u) \in [0,\infty)$ for all eigenvectors $u$ of $A$;
\item[iii)]  if $(u|u) = 0$ for some eigenvector $u$ of $A$, then $u \in \cln(A)$.
\end{itemize}
Then all eigenvalues of $A$ are nonnegative.
\end{lemma}
\begin{proof}
Let $\lambda \neq 0$ be a (possibly complex) eigenvalue of $A$. Then, if $u \in 
D(A)$ is a corresponding eigenvector, we have by ii) and iii)  that $(u|u) > 0$. 
The assertion follows from $\lambda (u|u) = (\lambda u|u) = (Au|u) \geq 0$.
\end{proof}
\begin{lemma} \label{RN0b}
Let $X$ be a Banach space, $N,R$ linear  subspaces with the properties that $N 
\cap R = \{ 0 \}$, $\mbox{dim}(N) = \mbox{codim}(R) = M < \infty$ and that $R$ 
is closed. Then the quotient map $Q: X \rightarrow X/R$, $y \mapsto y + R$ 
induces a topological isomorphism from $N$ onto $X/R$ and it vanishes on $R$. 
Moreover, $X = N \oplus R$ algebraically and topologically.
\end{lemma}
\begin{proof}
It straightforward to check that $Q|_N$  is a topological isomorphism onto the 
finite dimensional Banach space $X/R$ and that $Q$ vanishes on $R$. Moreover, $P 
:= (Q|_{N})^{-1} \circ Q$ is a continuous projection of $X$ onto $N$, hence
\[
X = N \oplus \mathcal{N}(P), \quad \mbox{dim}(X/\mathcal{N}(P)) = M. 
\]
Since $R \subset \mathcal{N}(P)$ and $\mbox{dim}(X/R) = M$, we  find 
$\mbox{dim}(\mathcal{N}(P)/R) = \mbox{dim}(X/R) - \mbox{dim}(X/\mathcal{N}(P)) = 
0$ and thus $\mathcal{N}(P) = R$.
\end{proof}
Since the range of a Fredholm operator is always closed, we get the following
\begin{corollary} \label{RN0c} 
If $X,Y$ are Banach spaces such that $Y \subset X$, and if $F \in 
\mathcal{L}(Y,X)$ is a Fredholm operator of index $0$ that satisfies 
$\mathcal{R}(F) \cap \mathcal{N}(F) = \{ 0 \}$, then $X = \mathcal{N}(F) \oplus  
\mathcal{R}(F)$ algebraically and topologically.  Moreover, $P := 
(Q|_{\mathcal{N}(F)})^{-1} \circ Q$ is a continuous projection of $X$ onto 
$\mathcal{N}(F)$.
\end{corollary}

\end{document}